\documentclass{amsart}
\usepackage{amsmath,amsthm}
\usepackage{amsfonts,amssymb}

\usepackage[english]{babel}
\usepackage{enumerate}

\usepackage{xcolor}

\newtheorem{theorem}{Theorem}
\newtheorem{proposition}[theorem]{Proposition}
\newtheorem{corollary}[theorem]{Corollary}
\newtheorem{lemma}[theorem]{Lemma}

\numberwithin{theorem}{section}

\theoremstyle{definition}
\newtheorem{definition}[theorem]{Definition}

\theoremstyle{remark}
\newtheorem{remark}[theorem]{Remark}

\newcommand{\mR}{\mathbb{R}}
\newcommand{\mC}{\mathbb{C}}
\newcommand{\mN}{\mathbb{N}}
\newcommand{\mE}{\mathbb{E}}

\newcommand{\mS}{\mathbb{S}}

\newcommand{\cM}{\mathcal{M}}

\newcommand{\cP}{\mathcal{P}}
\newcommand{\cC}{\mathcal{C}}
\newcommand{\cR}{\mathcal{R}}
\newcommand{\cK}{\mathcal{K}}

\newcommand{\cH}{\mathcal{H}}

\newcommand{\ux}{\underline{x}}

\newcommand{\uy}{\underline{y}}
\newcommand{\uu}{\underline{u}}
\newcommand{\uv}{\underline{v}}
\newcommand{\uz}{\underline{z}}
\newcommand{\us}{\underline{s}}
\newcommand{\ut}{\underline{t}}
\newcommand{\uom}{\underline{\omega}}
\newcommand{\uta}{\underline{\tau}}
\newcommand{\utd}{\underline{\tau}^{\dagger}}

\newcommand{\upx}{\partial_{\underline{x}}}
\newcommand{\upy}{\partial_{\underline{y}}}
\newcommand{\upz}{\partial_{\underline{z}}}

\newcommand{\pJz}{\partial_{z_j}}

\def\a{{\alpha}} 
\def\b{{\beta}}

\numberwithin{equation}{section}

\begin{document}
 
\title[Inversions for the Hua-Radon and polarized Hua-Radon transform]{Inversions for the Hua-Radon and the polarized Hua-Radon transform}

\author{Teppo Mertens}
\address{Department of Electronics and Information Systems \\Faculty of Engineering and Architecture\\Ghent University\\Krijgslaan 281, 9000 Gent\\ Belgium.}
\email{Teppo.Mertens@UGent.be}

\author{Frank Sommen}
\address{Department of Electronics and Information Systems \\Faculty of Engineering and Architecture\\Ghent University\\Krijgslaan 281, 9000 Gent\\ Belgium.}
\email{Franciscus.Sommen@UGent.be}

\date{\today}

\begin{abstract}
The Hua-Radon and polarized Hua-Radon transform are two orthogonal projections defined on holomorphic functions in the Lie sphere. Both transformations can be written as integral transforms with respect to a suitable reproducing kernel. Integrating both kernels over a Stiefel manifold yields a linear combination of zonal spherical monogenics. Using an Almansi type decomposition of holomorphic functions and reproducing properties of the zonal monogenics, we obtain an inversion formula for both the Hua-Radon and the polarized Hua-Radon transform.
\end{abstract}

\maketitle

\tableofcontents

\section{Introduction}
\setcounter{equation}{0}
The Szeg\H o-Radon transform was defined by Sabadini and Sommen in \cite{Szego} as an orthogonal projection of left monogenic functions onto a subspace $\mathcal{ML}^2(\uta)$. The Szeg\H o-Radon transform is a variant of the Clifford Radon transform which was established in \cite{Quaternionic and Radon, Radon to Clifford, Radon and X-ray, Clifford and integral}. Using techniques of Clifford analysis, a reproducing kernel for the basis of $\mathcal{ML}^2(\uta)$ was obtained. This kernel was then used to write the Szeg\H o-Radon transform as an integral transform over the $m$-dimensional unit sphere $\mS^{m-1}$. Finally, the dual Radon transform was defined as the integral over a Stiefel manifold of a function depending on two orthogonal unit vectors. This dual transform was the key to invert the Szeg\H o-Radon transform, since applying it termwise to the kernel of the Szeg\H o-Radon transform yields a scalar multiple of the zonal spherical monogenics. This scalar coefficient was reformulated as an operator using the Gamma operator $\Gamma_{\ux}$. This has lead to an inversion formula for the Szeg\H o-Radon transform.

Sabadini and Sommen recently extended the concept of the Szeg\H o-Radon transform to the Lie sphere \cite{radonlie}. This was motivated by the fact that monogenic functions admit a holomorphic extension in the Lie ball, see e.g. \cite{Hua, Mor,holomorphic}. Multiple kinds of mutually interrelated Radon type transforms were defined. These transforms were also reformulated as an integral transform with respect to a certain kernel, but this time the Lie sphere was considered as the integration domain.

A crucial development which has not yet appeared in the literature on these Radon type transforms is the description of their inversion. In this paper we will outline how inversion formulas can be obtained for two important Radon type transforms, namely the Hua–Radon and the polarized Hua–Radon transform. A similar analysis on yet another transform named after Hua, namely the monogenic Hua–Radon transform, was conducted by the authors in \cite{monogenic Hua}. We will use the techniques of \cite{Szego, monogenic Hua} also in the present setting.

We start our exposition with some preliminaries on Clifford analysis. In Section \ref{Section::Hua} we define Hua-Radon transform as an orthogonal projection onto a subspace of holomorphic functions over the Lie sphere. This orthogonal projection is then written as an integral transform with respect to a reproducing kernel $\cK_{\uta}$. In Proposition \ref{unicity harmonic}, we show that the zonal spherical harmonics $K_{m,k}(\ux,\uy)$ are spin-invariant, $k$-homogeneous and harmonic polynomials with respect to both $\ux$ and $\uy$ and moreover they are unique with this property up to a scalar multiple. Since each of the basis elements of $\mathcal{OL}^2(\uta)$ is a null-solution of some power of the complexified Laplace operator $\Delta_{\uz} = \sum_{j=1}^m \pJz^2$, we first show that the dual Radon transform applied to each of the terms of $\cK_{\uta}$ is a linear combination of the zonal spherical harmonics. Decomposing the zonal spherical harmonics in terms of the zonal spherical monogenics, we obtain an inversion formula for holomorphic functions over the Lie sphere in Theorem \ref{inversion Hua-Radon}. Quintessentially in this respect is the fact that holomorphic functions admit an Almansi decomposition \cite{Aros Almansi, Balk Almansi, Richard Almansi, Almansi, Lie}.

In Section \ref{section::polarized} we recall the definition of the polarized Hua-Radon transform, which is an orthogonal projection onto a subspace spanned by null-solutions of a power of the Dirac operator $\upz = \sum_{j=1}^{m}e_j\pJz$. Writing this projection as an integral transform with kernel $L_{\uta}$ and using the results of Section \ref{Section::Hua}, we can formulate the dual transform of each of the terms of $L_{\uta}$ as a linear combination of the zonal spherical monogenics $\cC_{m,k}(\ux,\uy)$. Finally, an inversion formula is determined in Theorem \ref{inversion polarized}, relying again on the Almansi decomposition for holomorphic functions.

\section{Preliminaries}
\setcounter{equation}{0}

In this section we introduce all notations and preliminary results that will be useful for the paper. We mostly follow the notations from \cite{radonlie}.

\subsection{Clifford algebras}

Let $\mR^{m}$ denote the real vector space with basis $(e_1, e_2, \ldots, e_{m})$. Throughout the paper we will assume that $m\geq 3$. We define the real Clifford algebra $\mR_{m}$ as the real algebra generated by the basis elements $e_1, e_2, \ldots, e_{m}$ which satisfy the following relations
\begin{align}
e_j^2 &= -1, \qquad j \in \{1, \ldots, m\}, \label{rel1}\\
e_j e_k + e_k e_j &=0, \phantom{-}\qquad j \neq k.\label{rel2}
\end{align}
The complex Clifford algebra $\mC_m$ is the complex algebra generated by $e_1, e_2, \ldots, e_{m}$ which satisfy (\ref{rel1}) and (\ref{rel2}).
Any element $\a\in\mR_m$ (respectively $\a\in\mC_m$), can be written as

\[
\a = \sum_{A\subset\{1,\ldots,m\}} \alpha_A e_A
\]
with $\a_A\in\mR$ (or $\a_A\in\mC$), $A = \{i_1,\ldots,i_{\ell}\}$ is a multi-index for which $i_1<\ldots <i_{\ell}$, such that $e_A = e_{i_1}\ldots e_{i_{\ell}}$ and $e_{\emptyset} = 1$. A 1-vector is a linear combination of only basis vectors $e_j$. We will denote the $k$-vector part of a real or complex Clifford element $\alpha$ by $[\alpha]_k$, i.e.

\[
[\alpha]_k = \sum_{\substack{A\subset\{1,\ldots,m\}\\|A|=k}} \alpha_{A} e_A
\]
We will be using the Hermitian conjugation, which is an automorphism on $\mC_m$ defined for $\a,\b\in\mC_m$ as

\begin{align*}
(\a\b)^{\dagger} &= \b^{\dagger}\a^{\dagger},\\
(\a + \b)^{\dagger} &= \a^{\dagger} + \b^{\dagger},\\
(\a_A e_A)^{\dagger} &= \overline{\a_A} e_A^{\dagger},\\
e_j^{\dagger}&=-e_j \qquad j\in\{1,\ldots,m\},
\end{align*}
where $\overline{\a_A}$ is the complex conjugate of $\a_A\in\mC$. Note that for $A = \{i_1,\ldots,i_k\}\subset \{1,\ldots,m\}$, we have
\begin{align}
e_A^{\dagger} e_A = (-1)^{k} e_{i_k}\ldots e_{i_1} e_{i_1}\ldots e_{i_k} = (-1)^{2k} = 1
\end{align}
and thus we have for any Clifford element $\alpha = \sum_{A\subset\{1,\ldots,m\}} \alpha_A e_A$ 

\begin{equation}\label{eq::0-part of fdag f}
\left[\alpha^{\dagger} \alpha\right]_0 = \sum_{A\subset\{1,\ldots,m\}} \overline{\alpha_A}\alpha_A.
\end{equation}

\subsection{Clifford analysis}

The scalar product of two 1-vectors $\uu$ and $\uv$ is given by $\langle \uu,\uv\rangle = \sum_{j=1}^m u_jv_j$ and the wedge-product by $\uu\wedge\uv = \sum_{i<j}(v_iw_j-v_jw_i)e_{i}e_{j}$. It is straightforward that
\begin{equation}\label{product vectors}
\uu\phantom{.}\uv = -\langle\uu,\uv\rangle + \uu\wedge\uv.
\end{equation}
Using equation (\ref{product vectors}) we see that $\langle \uu,\uv\rangle = -\frac{1}{2}(\uu\uv + \uv\uu)$ and $\uu\wedge\uv = \frac{1}{2}(\uu\uv - \uv\uu)$. Consequently if $\uu$ and $\uv$ are perpendicular, i.e. $\langle \uu,\uv\rangle = 0$, then $\uu\phantom{.}\uv = -\uv\phantom{.}\uu$.

The norm of a real or complex 1-vector $\uu$ is defined as
\[
|\uu|^2 = \sum_{j=1}^m u_j^2 = -\uu^2.
\]
If $\uu$ is a real vector, we can use $|\uu| = \sqrt{\sum_{j=1}^m u_j^2}$. If on the other hand $\uu$ is a complex vector, we use $|\uu|^2 = -\uu^2 = \langle \uu,\uu\rangle$, i.e. the square of the complexified version of the real norm.

The following result, proven in \cite{Szego}, will be useful.

\begin{lemma}\label{lemma::tau}
Let $\ut,\us\in\mS^{m-1}$ be such that $\langle\ut,\us\rangle = 0$ and let $\uta = \ut+i\us\in\mC^m$. Then $\utd = -\ut + i\us$ and
\begin{enumerate}[(i)]
\item
$\uta\utd\uta = 4\uta$,
\item
$\uta^2 = (\utd)^2 = 0$,
\item
$\uta\utd + \utd\uta = 4$.
\end{enumerate}
\end{lemma}

Let $B(0,1)$ denote the unit ball with center at the origin in $\mR^m$, whereas the unit sphere will be denoted by $\mS^{m-1}$, i.e. $\mS^{m-1} = \left\{\uu\in\mR^m \mid |\uu|^2 = 1\right\}$. The area of the unit sphere is given by
\[
A_m = \frac{2\pi^{m/2}}{\Gamma\left(\frac{m}{2}\right)},
\]
where $\Gamma$ is the gamma function.\\
The standard Dirac operator is given by
\[
\upx = \sum_{j=1}^m e_j\partial_{x_j}.
\]
Using (\ref{rel1}) and (\ref{rel2}), we get that the square of the Dirac operator satisfies $\upx^2= - \Delta_{\ux}$, where $\Delta_{\ux} = \sum\limits_{j=1}^{m}\partial_{x_{j}}^2$ is the Laplace operator. The symbol of the Dirac operator $\upx$ is denoted by the vector variable
\[
\ux= \sum_{j=1}^{m}e_{j}x_j
\]

\begin{definition}
A function $f:\Omega\subset \mR^m \to\mC_m$ which is continuously differentiable in the open set $\Omega$ is called (left) monogenic in $\Omega$ if $f$ is in the kernel of the Dirac operator $\upx$, i.e. $\upx f = \sum_{j=1}^m e_j (\partial_{x_j}f(\ux)) = 0$. The right $\mC_m$-module of (left) monogenic functions in $\Omega$ is denoted by $\cM(\Omega)$.\\
A function $f:\Omega\subset \mR^m \to\mC_m$ which is continuously differentiable in the open set $\Omega$ is called harmonic in $\Omega$ if $f$ is in the kernel of the Laplace operator $\Delta_{\ux} = -\upx^2$. The right $\mC_m$-module of harmonic functions in $\Omega$ is denoted by $\cH(\Omega)$.
\end{definition}
\begin{remark}
Note that there is also the notion of a right monogenic, i.e. a function $f$ for which
\[
f(\ux)\upx = \sum_{j=1}^m (\partial_{x_j} f(\ux))e_j = 0.
\]
\end{remark}
The Gamma-operator $\Gamma_{\ux}$ is given by
\[
\Gamma_{\ux} = -\ux\upx-\mE_{\ux},
\]
where $\mE_{\ux}$ is the Euler operator, which is defined as
\[
\mE_{\ux} = \sum_{j=1}^m x_j\partial_{x_j}.
\]
Let $\cP(\mR^m)$ be the space of polynomials in $m$ variables $x_1,\ldots, x_m$ with coefficients in $\mR$. Using the Euler operator we can define the space $\cP_k(\mR^m)$ of $k$-homogeneous polynomials as
\[
\cP_k(\mR^m) = \{P(\ux)\in\cP(\mR^m) \mid \mE_{\ux} P(\ux) = kP(\ux)\}
\]
which leads to the definition of the space of monogenic polynomials of degree $k$ 
\[
\cM_k(\mR^m) = \left(\cP_k(\mR^m)\otimes \mC_m\right) \cap \cM(\mR^{m})
\]
and the space of spherical harmonics of degree $k$ as

\[
\cH_k(\mR^m) = \left(\cP_k(\mR^m)\otimes \mC_m\right) \cap \cH(\mR^{m}).
\]
We will usually write $\cP_k$, $\cM_k$ and $\cH_k$ for simplicity.\\
\begin{remark}\label{Remark::complexify}
The vectors $\ux$ and $\uy$ are denoting real-valued variables, whereas $\uz$ will denote a complex-valued variable. This also applies to any of the operators defined above. Note that if we use a complex variable $\uz$ we complexify the operator. For example: we work with the complexified version of the Laplace operator
\[
\Delta_{\uz} = \sum_{j=1}^{m} \pJz^2.
\]
If we talk about monogenic or harmonic functions in the complex case, we mean null-solutions of the complexified Dirac operator $\upz$ or the complexified Laplace operator $\Delta_{\uz}$. The space of polynomials of degree $k$ monogenic with respect to $\upz$ is denoted by $\cM_k(\mC^m)$ and the corresponding space of harmonics by $\cH_k(\mC^m)$.
\end{remark}
The space of $k$-homogeneous polynomials $\cP_k$ can be decomposed in terms of these monogenic polynomials and harmonic polynomials in what is known as the harmonic Fischer decomposition (see e.g. \cite{Groen}):

\begin{equation}\label{eq::harmon fisch decomp}
\cP_k(\mR^m)\otimes \mR_m = \bigoplus_{j=0}^{\left\lfloor\frac{k}{2}\right\rfloor} |\ux|^{2j} \cH_{k-2j}.
\end{equation}
Moreover we can decompose $\cH_k$ into $\cH_k = \cM_k \oplus \ux \cM_{k-1}$ (see \cite{Fischer, Groen}).
\begin{proposition}\label{harmonics to monogenics}
Let $H_k(\ux) \in \cH_k$. Then we have $H_k(\ux)= M_k(\ux) + \ux M_{k-1}(\ux)$, with

\begin{align*}
M_{k-1}(\ux) &= -\frac{1}{2k+m-2}\upx H_k(\ux),\\
M_k(\ux)&= \left(1+\frac{1}{2k+m-2}\ux\upx\right) H_k(\ux)
\end{align*}
and $M_{k-1}\in \mathcal{M}_{k-1}$, $M_{k}\in \mathcal{M}_{k}$.
\end{proposition}
Combining Proposition \ref{harmonics to monogenics} and (\ref{eq::harmon fisch decomp}) leads to the monogenic Fischer decomposition:
\begin{equation}\label{eq::monogenic fisch decomp}
\cP_k(\mR^m)\otimes \mR_m = \bigoplus_{j=0}^{k} \ux^{j} \cM_{k-j}.
\end{equation}

For the decomposition in (\ref{eq::harmon fisch decomp}), there is a projection operator in order to determine each term of the decomposition (see e.g. \cite{projection}):

\begin{proposition}\label{prop projection}
The projection operator of a $k$-homogeneous polynomial onto its harmonic component of degree $k-2\ell$ is given by the following operator

\[
\mbox{Proj}_{k,l}=\sum_{j=0}^{\left\lfloor \frac{k}{2}\right\rfloor-\ell} \alpha_{j,k,\ell} | \ux|^{2j}\Delta_{\ux}^{j+\ell}
\]
where

\[
\alpha_{j,k,\ell} = \frac{(-1)^j(\frac{m}{2}+k-2\ell-1)}{4^{j+l}j!\ell!}\frac{\Gamma(\frac{m}{2}+k-2\ell-j-1)}{\Gamma(\frac{m}{2}+k-\ell)}.
\]
\end{proposition}
\subsection{Functions over the Lie Sphere}
The aim of the paper is to invert the Hua-Radon transform and the polarized Hua-Radon transform, which were defined in \cite{radonlie} as an orthogonal projection of functions over the Lie sphere. The definition of the Lie sphere can be found in e.g. \cite{Lie}. Let us recall the necessary definitions and results that we will need.

\begin{definition}
The Lie ball $LB(0,1)$ is defined as
\[
LB(0,1) = \{\uz=\ux+i\uy \in \mC^{m}\mid S_{\ux,\uy} \subset B(0,1)\}
\]
where $S_{\ux,\uy}$ is the codimension 2 sphere:

\[
S_{\ux,\uy} = \{\uu\in\mR^{m} \mid |\uu-\ux| = |\uy|, \langle\uu-\ux,\uy\rangle = 0\}.
\]
\end{definition}

\begin{definition}
The Lie sphere $LS^{m-1}$ is given by

\[
LS^{m-1} = \{e^{i\theta} \uom \in\mC^{m} \mid \uom\in S^{m-1}, \theta\in [0,\pi)\}.
\]
\end{definition}
We will be working with a space of holomorphic functions on the Lie sphere:

\begin{definition}
The right $\mC^m$-module containing all holomorphic functions $f:LB(0,1)\to\mC_m$ such that
\[
\left[\int_{\mS^{m-1}} \int_0^{\pi} \left[f(e^{i\theta}\uom)\right]^{\dagger} f(e^{i\theta}\uom) d\theta dS(\uom)\right]_0 < \infty
\]
will be denoted by $\mathcal{OL}^2 (LB(0,1))$.\\
The space $\mathcal{OL}^2 (LB(0,1))$ is equipped with the following inner product
\[
\langle f,g\rangle_{\mathcal{OL}^2 (LB(0,1))} = \int_{\mS^{m-1}}\int_{0}^{\pi}\left[f(e^{i\theta}\uom)\right]^{\dagger} g(e^{i\theta}\uom) d\theta dS(\uom).
\]
We will often write $\langle f,g\rangle_{\mathcal{OL}^2}$ for the inner product of $f$ and $g$.
\end{definition}
It was shown in \cite{Spherical and analytic, Lie} that any $f\in\mathcal{OL}^2(LB(0,1))$ admits an Almansi type decomposition:

\begin{equation}\label{eq::Almansi}
f(\uz) = \sum_{k=0}^{\infty} M_k(\uz) + \uz \sum_{\ell = 0}^{\infty} N_\ell(\uz)
\end{equation}
where $M_k \in \cM_k(\mC^m)$, $N_{\ell}\in \cM_{\ell}(\mC^m)$. A detailed account on the Almansi decomposition can be found in \cite{Aros Almansi, Balk Almansi, Richard Almansi, Almansi}.


\section{Inversion of the Hua-Radon transform}\label{Section::Hua}
\subsection{The Hua-Radon transform}
\setcounter{equation}{0}

Let us first recall the definition of the Hua-Radon transform which was introduced in \cite{radonlie}.\\
Let $\uta = \ut + i\us$ with $\ut,\us\in\mS^{m-1}$, $\langle\ut,\us\rangle = 0$ and let
\[
f_{\uta,k,\ell}(\uz) = \langle\uz, \uta\rangle^k\langle\uz, \utd\rangle^\ell.
\]
The closure of the $\mC_m$-submodule of $\mathcal{OL}^2(LB(0,1))$ generated by $\{f_{\uta,k,\ell}(\uz)\mid k,\ell\in\mN\}$ will be denoted by $\mathcal{OL}^2(\uta)$.\\
The functions $f_{\uta, k, \ell}(\uz)$ form an orthogonal basis for $\mathcal{OL}^2 (\uta)$ with respect to the inner product $\langle\cdot, \cdot\rangle_{\mathcal{OL}^2}$. We have the following result from \cite{radonlie}:
\begin{proposition}
Let $\uta = \ut + i\us$, where $\ut, \us\in\mS^{m-1}$ and $\langle\ut, \us\rangle = 0$. The functions $f_{\uta, k, \ell}(\uz)$ are such that
\[
\langle f_{\uta, k, \ell}, f_{\uta, k', \ell'}\rangle_{\mathcal{OL}^2} = 0 \qquad (k,\ell)\neq (k',\ell')
\]
and
\[
\langle f_{\uta, k, \ell}, f_{\uta, k, \ell}\rangle_{\mathcal{OL}^2} = 2\pi^{m/2+1} \frac{\Gamma(k+l+1)}{\Gamma(k+l+m/2)}.
\]
\end{proposition}
The Hua-Radon transform is defined as the orthogonal projection
\begin{align*}
\cH_{\uta}:&\mathcal{OL}^2(LB(0,1))\to \mathcal{OL}^2(\uta):\\
&f \mapsto \int_{\mS^{m-1}}\int_{0}^{\pi} \cK_{\uta}(\uz, e^{-i\theta}\uom) f(e^{i\theta}\uom) d\theta dS(\uom)
\end{align*}
with
\begin{align*}
\cK_{\uta}(\uz, e^{-i\theta}\uom) &= \frac{1}{\pi A_m} \sum_{k,\ell=0}^{\infty} (-1)^{k+\ell} \frac{\Gamma\left(k+\ell+\frac{m}{2}\right)}{\Gamma(k+\ell+1) \Gamma\left(\frac{m}{2}\right)} a^k b^{\ell}\\
&=\frac{1}{\pi A_m} \sum_{s=0}^{\infty}\sum_{k=0}^{s} (-1)^{s} \frac{\Gamma\left(s+\frac{m}{2}\right)}{\Gamma(s+1) \Gamma\left(\frac{m}{2}\right)} a^{s-k} b^{k},
\end{align*}
where $a = \langle\uz,\uta\rangle\langle e^{-i\theta}\uom,\utd\rangle$ and $b = \langle\uz,\utd\rangle\langle e^{-i\theta}\uom,\uta\rangle$. The function $\cK_{\uta}(\uz, e^{-i\theta}\uom)$ is the reproducing kernel of the functions $f_{\uta, k, \ell}(uz)$, i.e.
\[
\int_{\mS^{m-1}}\int_{0}^{\pi} \cK_{\uta}(\uz, e^{-i\theta}\uom) f_{\uta, k,\ell}(e^{i\theta}\uom) d\theta dS(\uom) = f_{\uta, k, \ell}(\uz)
\]
\subsection{Some technical results}
In order to invert the Hua-Radon transform, we will integrate its kernel $\cK_{\uta}$ over a Stiefel manifold, which leads to the definition of the dual Radon transform:

\begin{definition}
The dual Radon transform $\tilde{R}[F(\uz,\uta)]$ is defined as

\[
\tilde{R}[F(\uz,\uta)] = \frac{1}{A_{m}A_{m-1}}\int_{\mS^{m-1}}\left(\int_{\mS^{m-2}}F(\uz,\uta) dS(\us)\right)dS(\ut),
\]
where $\mS^{m-2}\subseteq \mS^{m-1}$ is the $(m-2)$-sphere orthogonal to $\ut$.
\end{definition}

We will also need some results of zonal spherical harmonics and zonal spherical monogenics. The zonal spherical harmonics are given by, see e.g. \cite{zonalmono}

\[
K_{m,k}(\ux,\uy) = \frac{2k+m-2}{m-2}|\ux|^k|\uy|^k C_k^{\frac{m}{2}-1}(t),
\]
where $t = \frac{\langle\ux,\uy\rangle}{|\ux||\uy|}$ and $C_k^{\frac{m}{2}-1}(t)$ is a Gegenbauer polynomial. These zonal spherical harmonics have the property of reproducing harmonic polynomials, i.e.
\begin{equation}\label{eq::repr harm}
\frac{1}{A_m}\int_{\mS^{m-1}} K_{m,j}(\ux,\uom) H_k(\uom) dS(\uom) = \delta_{j,k} H_x(\ux)
\end{equation}
where $H_k\in\cH_k$. The zonal spherical monogenics are given by (see \cite{zonalmono})
\begin{align*}
\cC_{m,k}(\ux,\uy) &= (|\ux| |\uy|)^k \left(\frac{k+m-2}{m-2} C_k^{\frac{m}{2}-1}(t) + \frac{\ux\wedge\uy}{|\ux||\uy|}C_{k-1}^{\frac{m}{2}}(t)\right)\\
&= (|\ux| |\uy|)^k \left(C_k^{\frac{m}{2}}(t) + \frac{\ux\uy}{|\ux||\uy|}C_{k-1}^{\frac{m}{2}}(t)\right).
\end{align*}
and likewise, they are the reproducing kernel of the monogenic polynomials, i.e. for $M_k\in\cM_k$ we have
\begin{equation}\label{eq::repr mono}
\frac{1}{A_m}\int_{\mS^{m-1}} \cC_{m,j}(\ux,\uom) M_k(\uom) dS(\uom) = \delta_{j,k} M_x(\ux).
\end{equation}
\begin{remark}\label{remark::reproducing properties}
Note that the reproducing properties (\ref{eq::repr harm}) and (\ref{eq::repr mono}) also hold when integrating over the Lie sphere, i.e.
\[
\frac{1}{\pi A_m}\int_{\mS^{m-1}}\int_{0}^{\pi}K_{m,j}(\uz, e^{-i\theta}\uom) H_k(e^{i\theta}\uom) d\theta dS(\uom) = \delta_{jk}H_k(\uz)
\]
and
\[
\frac{1}{\pi A_m}\int_{\mS^{m-1}}\int_{0}^{\pi}\cC_{m,j}(\uz, e^{-i\theta}\uom) M_k(e^{i\theta}\uom) d\theta dS(\uom) = \delta_{jk}M_k(\uz),
\]
see \cite[Remark 3.6]{radonlie}.
\end{remark}
In \cite{Szego} it is shown that these zonal spherical monogenics are, up to a constant factor $\lambda + \mu e_{1\ldots m}$ with $\lambda,\mu\in\mathbb{C}$, the unique $k$-homogeneous polynomials satisfying certain properties as mentioned in Proposition \ref{Unicity zonal mono}.
\begin{proposition}\label{Unicity zonal mono}
Let $F(\ux,\uy)$ be a Clifford valued function which is a homogeneous polynomial of degree $k$ in both $\ux$ and $\uy$. If $F$ is left-monogenic in $\ux$ and right-monogenic in $\uy$, i.e. $\upx F(\ux,\uy) = 0 = F(\ux,\uy)\upy $ and if $F(\ux,\uy)$ is also spin-invariant, i.e. for $\sigma\in\textup{Spin}(m)=\{\prod_{i=1}^{2s}\sigma_i \mid \sigma_i\in\mS^{m-1}\}$ we have
\[
\sigma F(\overline{\sigma}\ux\sigma,\overline{\sigma}\uy\sigma)\overline{\sigma} = F(x,y).
\]
Then there exist complex constants $\lambda$ and $\mu$ such that
\[
F(\ux,\uy) = (\lambda + \mu e_{1\ldots m})\cC_{m,k}(\ux,\uy).
\]
\end{proposition}
We have a similar result if the function is harmonic in both $\ux$ and $\uy$.

\begin{proposition}\label{unicity harmonic}
Let $F(\ux,\uy)$ be a $k$-homogeneous polynomial in $\ux$ and $\uy$ that is harmonic in both $\ux$ and $\uy$, i.e. $\Delta_{\ux} [F(\ux,\uy)] = 0 = \Delta_{\uy}[F(\ux,\uy)] $. If $F$ is also spin-invariant, i.e. for $\sigma\in\textup{Spin}(m)$ we have
\[
\sigma F(\overline{\sigma}\ux\sigma,\overline{\sigma}\uy\sigma)\overline{\sigma} = F(\ux,\uy),
\]
then there exists complex constants $\rho_1,\rho_2, \nu_1,\nu_2 \in \mC$ such that
\[
F(\ux,\uy) = (\rho_1+\nu_1 e_{1\ldots m}) K_{m,k}(\ux,\uy) + (\rho_2+\nu_2 e_{1\ldots m}) \cC_{m,k}(\ux,\uy).
\]
\end{proposition}
\begin{proof}
Using Proposition \ref{harmonics to monogenics} for both $\ux$ and $\uy$ we get

\[
F(\ux,\uy) = M_{k,k}(\ux,\uy) + \ux M_{k-1,k} (\ux,\uy) + M_{k,k-1}(\ux,\uy)\uy + \ux M_{k-1,k-1}(\ux,\uy) \uy
\]
where $M_{i,j}(\ux,\uy)$ is a left-monogenic, $i$-homogeneous function with respect to $\ux$ and right-monogenic, $j$-homogeneous function with respect to $\uy$. Since $F$ is spin-invariant, we get for each $\sigma\in\mbox{Spin}(m)$

\begin{align*}
F(\ux,\uy) =& \sigma F(\overline{\sigma}\ux\sigma,\overline{\sigma}\uy\sigma)\overline{\sigma}\\
=& \sigma M_{k,k}(\overline{\sigma}\ux\sigma,\overline{\sigma}\uy\sigma)\overline{\sigma} + \ux \sigma M_{k-1,k} (\overline{\sigma}\ux\sigma,\overline{\sigma}\uy\sigma)\overline{\sigma}\\
& + \sigma M_{k,k-1}(\overline{\sigma}\ux\sigma,\overline{\sigma}\uy\sigma)\overline{\sigma}\uy + \ux \sigma M_{k-1,k-1}(\overline{\sigma}\ux\sigma,\overline{\sigma}\uy\sigma)\overline{\sigma} \uy.
\end{align*}
as $\sigma\overline{\sigma} = 1$. In \cite{Gil Mur spin} it was shown that $\upx$ is spin-invariant and hence $\sigma M_{i,j}(\overline{\sigma}\ux\sigma,\overline{\sigma}\uy\sigma)\overline{\sigma}$ is monogenic. Now using the uniqueness of the decomposition in Proposition \ref{harmonics to monogenics}, we get
\[
\sigma M_{i,j}(\overline{\sigma}\ux\sigma,\overline{\sigma}\uy\sigma)\overline{\sigma} = M_{i,j}(\ux,\uy) \qquad i,j \in\{k,k-1\}.
\]
Hence the functions $M_{i,j}$ only depend on $|\ux|,|\uy|$ and $\langle \ux,\uy\rangle$, see \cite[Lemma 5.3]{Szego}. The latter are of course invariant under the transformation $\ux\mapsto -\ux,\ \uy\mapsto -\uy$, whereas on the other hand $M_{i,j}(\ux, \uy)$ is sent to $(-1)^{i+j}M_{i,j}(\ux,\uy)$ by this transformation. Consequently, we must have that $M_{k,k-1}(\ux,\uy) = M_{k-1,k}(\ux,\uy) = 0$, and thus we find
\[
F(\ux,\uy) = M_{k,k}(\ux,\uy) + \ux M_{k-1,k-1}(\ux,\uy) \uy
\]
where $M_{k,k}$ and $M_{k-1,k-1}$ are homogeneous polynomials, left-monogenic in $\ux$ and right-monogenic in $\uy$ and they are spin-invariant. Using Proposition \ref{Unicity zonal mono}, we get

\begin{align}\label{equa harmon zonal mono}
F(\ux,\uy) =& \alpha \cC_{m,k}(\ux,\uy) + \beta \ux \cC_{m,k-1} (\ux,\uy)\uy \nonumber\\
=& \alpha(|\ux| |\uy|)^k \left(C_k^{\frac{m}{2}}(t) + \frac{\ux\uy}{|\ux||\uy|}C_{k-1}^{\frac{m}{2}}(t)\right) \nonumber\\
&+ \beta\ux (|\ux| |\uy|)^{k-1} \left(C_{k-1}^{\frac{m}{2}}(t) + \frac{\ux\uy}{|\ux||\uy|}C_{k-2}^{\frac{m}{2}}(t)\right) \uy\nonumber\\
=& (\alpha+\beta) (|\ux| |\uy|)^k \left(C_k^{\frac{m}{2}}(t) + \frac{\ux\uy}{|\ux||\uy|}C_{k-1}^{\frac{m}{2}}(t)\right)\\
&- \beta (|\ux| |\uy|)^k \left(C_k^{\frac{m}{2}}(t) - C_{k-2}^{\frac{m}{2}}(t)\right)\nonumber\\
=& (\alpha+\beta) \cC_{m,k}(\ux,\uy) - \beta\frac{2k+m-2}{m-2}(|\ux| |\uy|)^kC_k^{\frac{m}{2}-1}(t)\nonumber\\
=&(\alpha+\beta) \cC_{m,k}(\ux,\uy) - \beta K_{m,k}(\ux,\uy),\nonumber
\end{align}
where $\alpha = a_1+b_1 e_{1\ldots n}$ and $\beta = a_2+b_2 e_{1\ldots n}$ with $a_1,a_2,b_1,b_2\in\mC$ and in the second to last line we used that (see e.g. \cite{zonalmono})

\[
C_k^{\frac{m}{2}}(t) - C_{k-2}^{\frac{m}{2}}(t) = \frac{2k+m-2}{m-2}C_k^{\frac{m}{2}-1}(t).
\]
\end{proof}

\begin{remark}\label{rem::scalar valued harmonic}
Note that if $F$ is a scalar valued function with the same properties as in Proposition \ref{unicity harmonic}, then $F(x,y) = \alpha K_{m,k}(x,y)$ with $\alpha\in\mC$.
\end{remark}
\begin{remark}\label{rem harmon to mono zonal}
Using the equations in (\ref{equa harmon zonal mono}), we can easily see that

\[
K_{m,k+1}(\ux,\uy) = \cC_{m,k+1}(\ux,\uy) - \ux \cC_{m,k}(\ux,\uy)\uy.
\]
\end{remark}
Using Proposition \ref{Unicity zonal mono} and Proposition \ref{unicity harmonic} we will prove that the dual Radon transform of each of the terms of the kernel of the Hua-Radon transform are equal to a linear combination of the zonal spherical harmonics. Recall that when we are working in the complex case, we are complexifying our operators. Hence it suffices to show this equality for real variables. Afterwards, we use Remark \ref{rem harmon to mono zonal} and the reproducing properties (\ref{eq::repr mono}) to find an inversion formula.\\
We have the following:

\begin{proposition}\label{prop Fischer decomp}
Suppose that $F(\ux,\uy)$ is a homogeneous polynomial of degree $k+\ell$ in both $\ux$ and $\uy$, then

\[
F(\ux,\uy) = \sum_{j=0}^{\lfloor\frac{k+\ell}{2}\rfloor}\sum_{j'=0}^{\lfloor\frac{k+\ell}{2}\rfloor} |\ux|^{2j} H_{k+\ell-2j,k+\ell-2j'}(\ux,\uy)|\uy|^{2j'}
\]
where $H_{k+\ell-2j,k+\ell-2j'}(\ux,\uy)$ is a spherical harmonic of degree $k+\ell-2j$ in $\ux$ and of degree $k+\ell-2j'$ in $\uy$.
\end{proposition}
\begin{proof}

The proof follows by applying the harmonic Fischer decomposition (\ref{eq::harmon fisch decomp}) first in $\ux$ and then in $\uy$.
\end{proof}
\begin{remark}
Let $F(\ux,\uy)$ be a homogeneous polynomial of degree $k+l$ in both $\ux$ and $\uy$. If $\Delta_{\ux}^{r+1} [F(\ux,\uy)] = \Delta_{\uy}^{s+1} [F(\ux,\uy)] = 0$, $\Delta_{\ux}^{r} [F(\ux,\uy)] \neq 0$ and $\Delta_{\uy}^{s} [F(\ux,\uy)] \neq 0$, then
\[
F(\ux,\uy) = \sum_{j=0}^{r}\sum_{j'=0}^{s} |\ux|^{2j} H_{k+\ell-2j,k+\ell-2j'}(\ux,\uy)|\uy|^{2j'}
\]
\end{remark}

\begin{lemma}\label{lemma laplacian}
For $j\leq \min\{k,\ell\}$ we have
\[
\Delta_{\ux}^{j} \left[\langle\ux,\uta\rangle^{k}\langle\ux,\utd\rangle^{\ell}\right] = (-4)^j \frac{k!}{(k-j)!}\frac{\ell!}{(\ell-j)!} \langle\ux,\uta\rangle^{k-j}\langle\ux,\utd\rangle^{\ell-j}.
\]
If $j>\min\{k,\ell\}$ then
\[
\Delta_{\ux}^{j} \left[\langle\ux,\uta\rangle^{k}\langle\ux,\utd\rangle^{\ell}\right] =0.
\]
\end{lemma}
\begin{proof}
Using Lemma \ref{lemma::tau} one can verify that

\begin{align*}
\Delta_{\ux}\left[\langle\ux,\uta\rangle^{k}\langle\ux,\utd\rangle^{\ell}\right] &= 2k\ell\langle\ux,\uta\rangle^{k-1}\langle\ux,\utd\rangle^{\ell-1} \langle\uta,\utd\rangle\\
&=-4k\ell\langle\ux,\uta\rangle^{k-1}\langle\ux,\utd\rangle^{\ell-1}.
\end{align*}
Hence repeating this process gives

\begin{equation}\label{Delta^j <>^k<>^l}
\Delta_{\ux}^{j} \left[\langle\ux,\uta\rangle^{k}\langle\ux,\utd\rangle^{\ell}\right] = (-4)^j \frac{k!}{(k-j)!}\frac{\ell!}{(\ell-j)!} \langle\ux,\uta\rangle^{k-j}\langle\ux,\utd\rangle^{\ell-j} .
\end{equation}
Supposing that $\min\{k,\ell\} = \ell<k$ then
\[
\Delta_{\ux}^{\ell+1}[\langle \ux,\uta\rangle^k \langle \ux,\utd \rangle^\ell] = (-4)^k \frac{k!\,\ell!}{(k-\ell)!} \Delta_{\ux}[\langle \ux,\uta \rangle^{k-\ell}]=0
\] since
\begin{align*}
\Delta_{\ux} \left[\langle\ux,\uta\rangle^{k-\ell}\right] &= (k-\ell)(k-\ell-1)\langle\ux,\uta\rangle^{k-\ell-2} \sum_{i=1}^m \tau_{i}^2\\
&= (k-\ell)(k-\ell-1)\langle\ux,\uta\rangle^{k-\ell-2} (-\uta^2)\\
&=0
\end{align*}
where we used Lemma \ref{lemma::tau}. If $\min\{k,\ell\} = k<\ell$, we get analogously,
\[
\Delta_{\ux}^{k+1}\left[\langle \ux,\uta\rangle^k \langle \ux,\utd \rangle^\ell\right] = 0.
\]
If $k= \ell$, then taking $j = k$ in equation (\ref{Delta^j <>^k<>^l}) yields a scalar. Letting $\Delta_{\ux}$ act one more time on equation (\ref{Delta^j <>^k<>^l}) yields zero. This proves our claim.
\end{proof}
\begin{remark}
A consequence of Lemma \ref{lemma laplacian} is that the functions $f_{\uta, k,\ell}(\uz)$ are null-solutions of $\Delta_{\uz}^{\min\{k,\ell\}+1}$.
\end{remark}

\begin{corollary}\label{cor}
Let $K_{m,j}$ be the zonal spherical harmonic of degree $j$ and let $k\geq l$. Then for $\uta=\ut+i\us$, with $\ut,\us\in\mS^{m-1}$ such that $\langle\ut,\us\rangle = 0$, there exists constants $\theta_{j,k,l}$ such that

\begin{align*}
\frac{1}{A_m A_{m-1}}\int_{\mathbb{S}^{m-1}}\int_{\mathbb{S}^{m-2}} \langle\ux,\uta\rangle^{k}\langle\ux,\utd\rangle^{\ell}&\langle\uy,\uta\rangle^{\ell}\langle\uy,\utd\rangle^{k}dS(\us)dS(\ut)\\
&= \sum_{j=0}^{\ell} \theta_{j,k,\ell} |\ux|^{2j} K_{m,k+\ell-2j}(\ux,\uy)|\uy|^{2j}.
\end{align*}
\end{corollary}
\begin{proof}
If we define the integral on the left-hand side in Corollary \ref{cor} as $L_{k,\ell}(\ux,\uy)$, then we see that this integral is homogeneous of degree $k+\ell$ in both $\ux$ and $\uy$. Using Lemma \ref{lemma laplacian} we see that $L_{k,\ell}$ is a solution of $\Delta^{\ell+1}[F]=0$ and not of $\Delta^{\ell}[F]=0$. Hence applying Proposition \ref{prop Fischer decomp} yields

\[
L_{k,\ell}(\ux,\uy) = \sum_{j=0}^{\ell}\sum_{j'=0}^{\ell}  |\ux|^{2j} H_{k+\ell-2j,k+\ell-2j'}(\ux,\uy)|\uy|^{2j'}.
\]
If we interchange $\ux$ and $\uy$, one can easily see that $L_{k,\ell}(\uy,\ux)  = \left[L_{k,\ell}(\ux,\uy)\right]^{\dagger}$ and therefore
\begin{align*}
\sum_{j=0}^{\ell}\sum_{j'=0}^{\ell}  |\uy|^{2j} &H_{k+\ell-2j,k+\ell-2j'}(\uy,\ux)|\ux|^{2j'} \\
&= \sum_{j=0}^{\ell}\sum_{j'=0}^{\ell} \left( |\ux|^{2j} H_{k+\ell-2j,k+\ell-2j'}(\ux,\uy)|\uy|^{2j'}\right)^{\dagger}\\
&=\sum_{j=0}^{\ell}\sum_{j'=0}^{\ell} |\uy|^{2j} \left[H_{k+\ell-2j,k+\ell-2j'}(\ux,\uy)\right]^{\dagger}|\ux|^{2j'}.
\end{align*}
By uniqueness of the Fischer decomposition, we get
\[
H_{k+\ell-2j,k+l-2j'}(\uy,\ux) = \left[H_{k+\ell-2j,k+l-2j'}(\ux,\uy)\right]^{\dagger}.
\]
But the left-hand side is a polynomial of degree $k+\ell-2j$ in $\uy$ and of degree $k+\ell-2j'$ in $\ux$, whereas the right-hand side  is a polynomial of degree $k+\ell-2j'$ in $\uy$ and of degree $k+\ell-2j$ in $\ux$. Consequently, $H_{k+\ell-2j,k+\ell-2j'}(\uy,\ux) = 0$ whenever $j\neq j'$. Hence the remaining decomposition is
\[
L_{k,\ell}(\ux,\uy) = \sum_{j=0}^{\ell} |\ux|^{2j} H_{k+\ell-2j,k+\ell-2j}(\ux,\uy)|\uy|^{2j}.
\]
Furthermore, let $\sigma\in \text{Spin}(m)$. First of all note that $\mbox{Spin}(m)/\{-1,1\} \simeq SO(m)$ and $\sigma\ux\overline{\sigma}$ is just a rotation of the vector $\ux$, see \cite{Friederich spin, Gil Mur spin}. Hence

\begin{align*}
\sigma L_{k,\ell}(\overline{\sigma}\ux\sigma, \overline{\sigma}\uy\sigma)\overline{\sigma} &= \int_{\mathbb{S}^{m-1}}\int_{\mathbb{S}^{m-2}} \langle\ux,\sigma\uta\overline{\sigma}\rangle^{k}\langle\ux,\sigma\utd\overline{\sigma}\rangle^{\ell}\langle\uy,\sigma\uta\overline{\sigma}\rangle^{\ell}\langle\uy,\sigma\utd\overline{\sigma}\rangle^{k}dS(\us)dS(\ut)\\
&= \int_{\mathbb{S}^{m-1}}\int_{\mathbb{S}^{m-2}} \langle\ux,\uta\rangle^{k}\langle\ux,\utd\rangle^{\ell}\langle\uy,\uta\rangle^{\ell}\langle\uy,\utd\rangle^{k}dS(\us)dS(\ut).
\end{align*}
Hence $L_{k,\ell}(\ux,\uy)$ is Spin-invariant, which implies that $H_{k+\ell-2j,k+\ell-2j}(\ux,\uy)$ is Spin-invariant for each $j$. Note that the integral $L_{k,l}$ is scalar valued. This means that, using Remark \ref{rem::scalar valued harmonic}, we can find constants $\theta_{j,k,\ell} \in \mC$ such that 

\[
H_{k+l-2j,k+\ell-2j}(\ux,\uy) = \theta_{j,k,\ell} K_{m,k+\ell-2j}(\ux,\uy).
\]
\end{proof}

\begin{remark}
Note that if $\ell\geq k$ in Corollary \ref{cor}, then the result stays the same, but the upper limit of the sum will be $k$ instead of $\ell$.
\end{remark}

The following result from \cite{Szego} will be useful to calculate $\theta_{n,k,\ell}$.
\begin{lemma}\label{Szego result}
Let $\uom\in\mS^{m-1}$ and $m\geq 3$, then
\begin{align*}
\frac{1}{A_m A_{m-1}}\int_{\mS^{m-1}}\int_{\mS^{m-2}} \langle\uom,\uta\rangle^{j}\langle\uom,\utd\rangle^{j}dS(\us)dS(\ut) &= \frac{(-1)^{j}A_{m-2}\pi}{A_m} \frac{\Gamma(\frac{m}{2}-1)\Gamma(j+1)}{\Gamma(\frac{m}{2}+j)}\\
&=(-1)^{j}\frac{\Gamma(\frac{m}{2})\Gamma(j+1)}{\Gamma(\frac{m}{2}+j)}.
\end{align*}
\end{lemma}

\begin{proposition}\label{prop coeffi}
The coefficients $\theta_{n,k,\ell}$ are given by:

\[
\theta_{n,k,\ell} = (-1)^{k+\ell} \frac{\Gamma(m-1)\Gamma\left(\frac{m}{2} + k-n-1\right)\Gamma\left(\frac{m}{2}+\ell-n-1\right)(m-2)}{4(\ell-n)!\Gamma\left(\frac{m}{2}+k+\ell-n\right)^2\Gamma(m+k+\ell-2n-2)}\left(\frac{k!\ell!}{n!}\right)^2 \frac{(k+\ell-2n)!}{(k-n)!}.
\]
\end{proposition}
\begin{proof}
Let $L_{k,\ell}$ denote the left-hand side of the equation in Corollary \ref{cor} and assume $\ell\leq k$. By Lemma \ref{lemma laplacian}, applying $\Delta_{\ux}^{j}$ with $j\leq \ell$ yields

\begin{align*}
\Delta_{\ux}^{j} \left[L_{k,\ell}\right] =& (-4)^{j} \frac{k!}{(k-j)!}\frac{l!}{(\ell-j)!} \frac{1}{A_m A_{m-1}}\\
&\times\int_{\mS^{m-1}}\int_{\mS^{m-2}}\langle\ux,\uta\rangle^{k-j}\langle\ux,\utd\rangle^{\ell-j}\langle\uy,\uta\rangle^{\ell}\langle\uy,\utd\rangle^{k} dS(\us)dS(\ut)
\end{align*}
and $\Delta_{\ux}^{j}\left[L_{k,l}(\ux,\uy)\right] = 0$ when $j>l$. Hence by Proposition \ref{prop projection} we have

\begin{align*}
\sum_{j=0}^{l-n} \alpha_{j,k+\ell, n} |\ux|^{2j}\Delta_{\ux}^{j+n}\left[L_{k,\ell}(\ux,\uy)\right] = \theta_{n,k,\ell} K_{m,k+\ell-2n}(\ux,\uy)|\uy|^{2n}.
\end{align*}
As $\theta_{n,k,\ell}$ is the same for all $\ux,\uy$, it suffices to consider the case where $\ux=\uy=\uom\in\mS^{m-1}$. Using the fact that, see e.g. \cite{NIST},

\[
C_{r}^{\frac{m}{2}-1} (1) = \frac{\Gamma(m-2+r)}{\Gamma(m-2)\Gamma(r+1)}
\]
we get

\begin{align*}
\sum_{j=0}^{\ell-n} \alpha_{j, k+\ell, n}|\ux|^{2j} \Delta_{\ux}^{j+n}\left[L_{k,\ell}(\ux,\uy)\right]\vert_{\ux=\uy=\uom} = \theta_{n,k,\ell} \frac{2(k+\ell-2n)+m-2}{m-2}\frac{\Gamma(m-2+k+\ell-2n)}{\Gamma(m-2)\Gamma(k+\ell-2n+1)}.
\end{align*}
Now using Lemma \ref{Szego result} we get

\begin{align}\label{sum hua radon}
\sum_{j=0}^{\ell-n} \alpha_{j,k+l,n} |\ux|^{2j}\Delta_{\ux}^{j+n}\left[L_{k,\ell}(\ux,\uy)\right]\vert_{\ux=\uy=\uom} =& \sum_{j=0}^{\ell-n} \frac{(-1)^j (\frac{m}{2}+k+\ell-2n-1)}{4^{j+n}j!n!} \frac{\Gamma(\frac{m}{2}+k+\ell-2n-j-1)}{\Gamma(\frac{m}{2}+k+\ell-n)}\nonumber\\
& \times \frac{(-4)^{j+n}k!\ell!}{(k-j-n)(\ell-j-n)!} \frac{(-1)^{k+\ell-j-n} \Gamma(\frac{m}{2}) \Gamma(k+\ell-j-n+1)}{\Gamma(\frac{m}{2}+k+\ell-j-n)}\nonumber\\
=& \frac{(\frac{m}{2}+k+\ell-2n-1)\Gamma(\frac{m}{2})(-1)^{k+\ell}k!\ell!}{n!\Gamma(\frac{m}{2}+k+\ell-n)}\nonumber\\
&\times \sum_{j=0}^{\ell-n}(-1)^j\frac{\Gamma(\frac{m}{2}+k+\ell-2n-j-1)\Gamma(k+\ell-j-n+1)}{j!\Gamma(\frac{m}{2}+k+\ell-j-n)(k-j-n)!(\ell-j-n)!}.
\end{align}
The summation (\ref{sum hua radon}) is actually a multiple of a hypergeometric function ${}_3F_2$:

\begin{align*}
\sum_{j=0}^{\ell-n}(-1)^j&\frac{\Gamma(\frac{m}{2}+k+\ell-2n-j-1)\Gamma(k+\ell-j-n+1)}{j!\Gamma(\frac{m}{2}+k+\ell-j-n)(k-j-n)!(\ell-j-n)!} =\\
& \frac{\Gamma(\frac{m}{2}+k+\ell-2n-1)\Gamma(k+\ell-n+1)}{\Gamma(\frac{m}{2}+k+\ell-n)(k-n)!(\ell-n)!}{}_3F_2([-p,a,b],[c,d];1)
\end{align*}
where

\begin{align*}
p &= \ell-n,\\
a &= -\frac{m}{2}-k-\ell+n+1,\\
b &= -k+n,\\
c &= -k-\ell+n,\\
d &= -\frac{m}{2}-k-\ell+2n+2,\\
  &= a+b-c+1-p.
\end{align*}
Using the Pfaff-Saalsch\"utz Balanced sum, see \cite{NIST}, we find

\begin{align*}
{}_3F_2([-p,a,b],[c,d];1) &= \frac{(c-a)_p(c-b)_p}{(c)_p(c-a-b)_p}\\
&= \frac{(\frac{m}{2}-1)_{\ell-n}(-\ell)_{\ell-n}}{(-k-\ell+n)_{\ell-n}(\frac{m}{2}+k-n-1)_{\ell-n}}\\
&= \frac{\Gamma(\frac{m}{2}+\ell-n-1)}{\Gamma(\frac{m}{2}-1)}\frac{\Gamma(\frac{m}{2}+k-n-1)}{\Gamma(\frac{m}{2}+k+\ell-2n-1)}\frac{\ell!}{n!}\frac{k!}{(k+\ell-n)!}.
\end{align*}
Now combining everything yields the result stated in the theorem. The case $k\leq \ell$ is the same as interchanging $k \leftrightarrow \ell$ in the above calculations. Note that due to the symmetry of $\theta_{n,k,l}$ with respect to $k$ and $\ell$, the result does not change.
\end{proof}

\subsection{The inversion}
We now have the necessary tools to find an inversion formula for the Hua-Radon transform.
\begin{theorem}\label{inversion Hua Radon}
Let $f\in\mathcal{OL}^2(LB(0,1))$ be of the form

\[
f(\uz) = \uz^{2n} H_{t-2n}(\uz)
\]
where $H_{t-2n}\in\cH_{t-2n}(\mC^m)$ is a spherical harmonic of degree $t-2n \geq 0$. Then

\[
\tilde{R}[\mathcal{H}_{\uta}[f]](\uz) = \varphi_{t,n} f(\uz)
\]
with

\begin{align*}
\varphi_{t,n} =&\dfrac{\Gamma(t+\frac{m}{2})\Gamma(m-1)(t-n)!^2\Gamma(t-2n+\frac{m}{2}-1)}{2 t! \Gamma(\frac{m}{2}+t-n)^2 \Gamma(m+t-2n-2)}\\
&\times {}_4F_3([n+1,n+1,2n-t,\frac{m}{2}-1],[n-t,n-t,2n-t-\frac{m}{2}+2];1).
\end{align*}
\end{theorem}
\begin{proof}
We can write the kernel of the Hua-Radon transform as

\[
\cK_{\uta}(\uz, e^{-i\theta} \uom) = \frac{1}{\pi A_{m}} \sum_{s=0}^{\infty} \sum_{k=0}^s (-1)^s \frac{\Gamma(s+\frac{m}{2})}{\Gamma(\frac{m}{2})\Gamma(s+1)}a^{s-k} b^{k}
\]
where $a = \langle \uz,\uta\rangle\langle e^{-i\theta} \uom,\utd\rangle$ and $b = \langle \uz,\utd\rangle\langle e^{-i\theta} \uom,\uta\rangle$. Thus using Corollary \ref{cor} we have

\begin{align}\label{Dual Hua-Radon}
\tilde{R}[\mathcal{H}_{\uta}[f]](\uz) =& \frac{1}{\pi A_{m}^2 A_{m-1}} \sum_{s=0}^{\infty} \sum_{k=0}^s (-1)^s \frac{\Gamma(s+\frac{m}{2})}{\Gamma(\frac{m}{2})\Gamma(s+1)} \nonumber\\
&\times \int_{\mS^{m-1}}\int_{\mS^{m-2}}\int_{\mS^{m-1}}\int_0^{\pi} a^{s-k} b^{k} (e^{i\theta}\uom)^{2n}H_{t-2n}(e^{i\theta}\uom) d\theta dS(\uom)dS(\us)dS(\ut)\nonumber\\
=&\frac{1}{\pi A_{m}} \sum_{s=0}^{\infty} \sum_{k=0}^s \sum_{n'=0}^{\min\{k,s-k\}} \theta_{n',s-k,k} (-1)^s \frac{\Gamma(s+\frac{m}{2})}{\Gamma(\frac{m}{2})\Gamma(s+1)}\\
&\times \int_{\mS^{m-1}}\int_0^{\pi} \uz^{2n'}K_{m,s-2n'}(\uz,e^{-i\theta}\uom)(e^{-i\theta}\uom)^{2n'}  (e^{i\theta}\uom)^{2n}H_{t-2n}(e^{i\theta}\uom) d\theta dS(\uom).\nonumber
\end{align}
If we now use the homogeneity of the zonal spherical harmonics and of $H_{t-2n}$, we can rewrite the last integral in (\ref{Dual Hua-Radon}) as follows

\begin{align*}
&\left(\int_0^{\pi} e^{i\theta(t-s)} d\theta\right)\int_{\mS^{m-1}} \uz^{2n'}K_{m,s-2n'}(\uz,\uom)(\uom)^{2(n'+n)}H_{t-2n}(\uom)  dS(\uom)\\
=&\left(\int_0^{\pi} e^{i\theta(t-s)} d\theta\right)\int_{\mS^{m-1}} (-1)^{n'+n}\uz^{2n'}K_{m,s-2n'}(\uz,\uom)H_{t-2n}(\uom)  dS(\uom).
\end{align*}
Therefore, if we use the reproducing property of the zonal spherical harmonic, we get that the above integral vanishes except when $s-2n' = t-2n$. But this means that $2n-2n' = t-s$ and thus, by looking at the integral over $\theta$, we get that the only non-vanishing integral occurs when $n'=n$ and consequently $s=t$. As $n=n'\leq \min\{k,t-k\}$, we have $k\geq n$ and $k\leq t-n$. So now (\ref{Dual Hua-Radon}) becomes
\begin{align*}
\tilde{R}[\mathcal{H}_{\uta}[f]](\uz) &= f(\uz)\sum_{k=n}^{t-n} \theta_{n,t-k,k} (-1)^t \frac{\Gamma(t+\frac{m}{2})}{\Gamma(\frac{m}{2})\Gamma(t+1)}\\
&= f(\uz)\psi_{t,n} \sum_{k=n}^{t-n} \xi_{n,t-k,k}  
\end{align*}
where

\begin{align*}
\psi_{t,n} &= \frac{\Gamma(t+\frac{m}{2})}{\Gamma(\frac{m}{2})t!}\frac{(m-2)\Gamma(m-1)\Gamma(t-2n+1)}{4\Gamma(\frac{m}{2}+t-n)^2\Gamma(m+t-2n-2)(n!)^2}\\
\xi_{n,t-k,k} &= \frac{((t-k)!)^2(k!)^2 \Gamma(t-k-n+\frac{m}{2}-1)\Gamma(k-n+\frac{m}{2}-1)}{(k-n)!(t-k-n)!}.
\end{align*}
We can rewrite this last summation as a multiple of a hypergeometric function:

\begin{align*}
\sum_{k=n}^{t-n} \xi_{n,t-k,k} =& \dfrac{(t-n)!^2n!^2\Gamma(t-2n+\frac{m}{2}-1)\Gamma(\frac{m}{2}-1)}{(t-2n)!}\\
&\times {}_4F_3([n+1,n+1,2n-t,\frac{m}{2}-1],[n-t,n-t,2n-t-\frac{m}{2}+2];1).
\end{align*}
Hence if we combine everything we get

\begin{align*}
\varphi_{t,n} =&\dfrac{\Gamma(t+\frac{m}{2})\Gamma(m-1)(t-n)!^2\Gamma(t-2n+\frac{m}{2}-1)}{2 t! \Gamma(\frac{m}{2}+t-n)^2 \Gamma(m+t-2n-2)}\\
&\times {}_4F_3([n+1,n+1,2n-t,\frac{m}{2}-1],[n-t,n-t,2n-t-\frac{m}{2}+2];1).
\end{align*}
\end{proof}

We can now use the fact that $M_t(\uz)$ and $\uz M_{t-1}(\uz)$ are both spherical harmonics, where $M_t(\uz)\in\cM_t(\mC^m), M_{t-1}(\uz)\in\cM_{t-1}(\mC^m)$ are spherical monogenics, to rewrite the result of Theorem \ref{inversion Hua Radon} as

\begin{align}
\tilde{R}[\mathcal{H}_{\uta}[f]](\uz) &= \varphi_{t,n} f(\uz) & f(\uz) &=\uz^{2n} M_{t-2n}(\uz), \label{even power}\\
\tilde{R}[\mathcal{H}_{\uta}[h]](\uz) &= \varphi_{t,n} h(\uz) & h(\uz) &=\uz^{2n+1} M_{t-2n-1}(\uz).\label{odd power}
\end{align}
In doing so we can use 

\begin{align}
\mE_{\uz}\left[\uz^{j} M_{t-j}(\uz)\right] &= t\uz^{j} M_{t-j}(\uz),\nonumber\\
\Gamma_{\uz} \left[\uz^{2n} M_{t-2n}(\uz)\right] &= (2n-t)\uz^{2n} M_{t-2n}(\uz),\label{eq::Euler,Gamma}\\
\Gamma_{\uz} \left[\uz^{2n+1} M_{t-2n-1}(\uz)\right] &= (t-2n+m-2)\uz^{2n+1} M_{t-2n-1}(\uz).\nonumber
\end{align}
Thus we can write $\varphi_{t,n} = \varphi_{\mathbb{E}_{\uz},\frac{\Gamma_{\uz} + \mE_{\uz}}{2}}$ in case (\ref{even power}) and $\varphi_{t,n} = \varphi_{\mathbb{E}_{\uz},\frac{\mE_{\uz}-\Gamma_{\uz}+m-2}{2}}$ in case (\ref{odd power}) so that it becomes an operator. Hence we have

\begin{align*}
\varphi_{\mathbb{E}_{\uz},\frac{\Gamma_{\uz} + \mE_{\uz}}{2}}^{-1}\tilde{R}[\mathcal{H}_{\uta}[g]](\uz) &= g(\uz) & g(\uz) &=\uz^{2n} M_{t-2n}(\uz), \\
\varphi_{\mathbb{E}_{\uz},\frac{\mE_{\uz}-\Gamma_{\uz}+m-2}{2}}^{-1}\tilde{R}[\mathcal{H}_{\uta}[h]](\uz) &= h(\uz) & h(\uz) &=\uz^{2n+1} M_{t-2n-1}(\uz).
\end{align*}
Since the Hua-Radon transform is defined for holomorphic functions, we can use (\ref{eq::Almansi}), so that we can invert it on any holomorphic function.

\begin{theorem}\label{inversion Hua-Radon}
Let $f(\uz) = \sum_{k=0}^{\infty} M_k(\uz) + \uz\sum_{\ell = 0}^{\infty} N_\ell(\uz) \in\mathcal{OL}^2(LB(0,1))$, with $M_k \in\cM_k(\mC^m)$, $N_{\ell}\in\cM_{\ell}(\mC^m)$, then
\[
f(\uz) = \varphi_{\mathbb{E}_{\uz},\frac{\Gamma_{\uz} + \mE_{\uz}}{2}}^{-1}\tilde{R}\left[\mathcal{H}_{\uta}\left[\sum_{k=0}^{\infty} M_k(\uz')\right]\right](\uz) + \varphi_{\mathbb{E}_{\uz},\frac{\mE_{\uz}-\Gamma_{\uz}+m-2}{2}}^{-1}\tilde{R}\left[\mathcal{H}_{\uta}\left[\uz'\sum_{\ell = 0}^{\infty} N_\ell(\uz')\right]\right] (\uz).
\]
\end{theorem}


\section{Inversion of the polarized Hua-Radon transform}\label{section::polarized}
\setcounter{equation}{0}
\subsection{The polarized Hua-Radon transform}
In this section we determine the inversion of the polarized Hua-Radon transform which was defined in \cite{radonlie} as follows.\\
Let $\uta = \ut + i\us$ with $\ut,\us\in\mS^{m-1}$, $\langle\ut,\us\rangle = 0$ and let
\begin{align*}
\psi_{\uta,2r,k}(\uz) &= \uta\langle\uz, \uta\rangle^{r+k}\langle\uz, \utd\rangle^{r} = \uta f_{\uta,r+k,r}(\uz),\\
\psi_{\uta,2r+1,k}(\uz) &= \utd\uta\langle\uz, \uta\rangle^{r+k+1}\langle\uz, \utd\rangle^{r} = \utd\uta f_{\uta,r+k+1,r}(\uz).
\end{align*}
One can show that the functions $\psi_{\uta, \alpha,k}(\uz)$ are null-solutions of $\upz^{\alpha+1}$ (see \cite{radonlie}).\\
For $s\in\mN$, we define the right $\mC_m$-submodule $\mathfrak{M}^s(\uta)$ of $\mathcal{OL}^2(LB(0,1))$ to be the completion of the space generated by $\{\psi_{\uta,s,k}(\uz)\mid k\in\mN\}$. The spaces $\mathfrak{M}^s(\uta)$ are orthogonal with respect to $\langle\cdot, \cdot \rangle_{\mathcal{OL}^2}$. The space $\mathfrak{M}(\uta)$ is defined as the direct orthogonal sum $\oplus_{s=0}^{\infty} \mathfrak{M}^s(\uta)$ (see \cite{radonlie}).
\begin{remark}
The space $\mathfrak{M}(\uta)$ was constructed in such a way that we can decompose $\mathcal{OL}^2(\uta)$ into
\[
\mathcal{OL}^2(\uta) = \mathfrak{M}(\uta) \oplus \mathfrak{M}(\utd)
\]
\end{remark}
The polarized Hua-Radon transform is defined as the orthogonal projection
\begin{align*}
\cR^H_{\uta}:&\mathcal{OL}^2(LB(0,1))\to \mathfrak{M}(\uta):\\
&f \mapsto \int_{\mS^{m-1}}\int_{0}^{\pi} L_{\uta}(\uz, e^{-i\theta}\uom) f(e^{i\theta}\uom) d\theta dS(\uom)
\end{align*}
with $L_{\uta}(\uz, e^{-i\theta}\uom)$ given by
\begin{align}
& \frac{1}{\pi A_m}\sum_{s=0}^{\infty}\sum_{k=0}^{\infty} (-1)^k \frac{\Gamma(k+2s+\frac{m}{2})}{\Gamma(k+2s+1)\Gamma(\frac{m}{2})}a^{k+s}b^{s} \label{kern sum 1}\\
&- \frac{1}{\pi A_m}\frac{\utd\uta}{4}\sum_{s=0}^\infty \frac{\Gamma(2s+\frac{m}{2})}{\Gamma(2s+1)\Gamma(\frac{m}{2})}a^s b^s \label{kern sum 2}
\end{align}
where $a=\langle\uz,\uta\rangle\langle e^{-i\theta}\uom,\utd\rangle$,  $b=\langle\uz,\utd\rangle\langle e^{-i\theta}\uom,\uta\rangle$. The functions $\psi_{\uta, \alpha, k}(\uz)$ are reproduced by the kernel $L_{\uta}(\uz, e^{-i\theta}\uom)$.

\subsection{The dual transform of the kernel}
In order to invert the polarized Hua-Radon transform, we will integrate its kernel over a Stiefel manifold aiming to get a linear combination of the zonal spherical monogenics. Recall that we are trying to use the reproducing properties of the zonal spherical monogenics to find an inversion formula for the polarized Hua-Radon transform. These reproducing properties also hold when working over the Lie sphere, see Remark \ref{remark::reproducing properties}. Since this is just a complexified version of the zonal spherical monogenics, it suffices to consider the kernel evaluated in real variables and show equality in the real case.\\
In order to evaluate the action of the dual Radon transform of (\ref{kern sum 1}), we can use the Corollary \ref{cor}, obtained for the Hua-Radon transform in Section \ref{Section::Hua}, and Proposition \ref{harmonics to monogenics}.
For (\ref{kern sum 2}), we will consider the action of the dual Radon transform on each term independently. We have for a term $a^s b^s$
\begin{equation}\label{eq::int Stiefel polarized}
\frac{1}{A_m A_{m-1}}\int_{\mS^{m-1}}\int_{\mS^{m-2}} \langle\ux,\uta\rangle^{k}\langle\ux,\utd\rangle^{k}\langle\uy,\uta\rangle^{k}\langle\uy,\utd\rangle^{k}\utd\uta dS(\us)dS(\ut)
\end{equation}
where $\mS^{m-2}$ is the unit sphere perpendicular to $\ut$. Now consider $\langle\ux,\uta\rangle^{k}\langle\ux,\utd\rangle^{k}\langle\uy,\uta\rangle^{k}\langle\uy,\utd\rangle^{k}\utd\uta$ as a function $f$ of the variable $\us$. As $\utd\uta = 2-i\ut\us$, we can write $f(\us) = f_0(\us)+f_2(\us)$, where
\[
f_0(\us) = 2\langle\ux,\uta\rangle^{k}\langle\ux,\utd\rangle^{k}\langle\uy,\uta\rangle^{k}\langle\uy,\utd\rangle^{k}
\]
is the scalar part of $f(\us)$ and
\[
f_2(\us) = -2i\langle\ux,\uta\rangle^{k}\langle\ux,\utd\rangle^{k}\langle\uy,\uta\rangle^{k}\langle\uy,\utd\rangle^{k}\ut\us 
\]
is its bivector part. We have the following:

\begin{align*}
f_2(-\us) &= -2i\langle\ux,\ut-i\us\rangle^{k}\langle\ux,-\ut-i\us\rangle^{k}\langle\uy,\ut-i\us\rangle^{k}\langle\uy,-\ut-i\us\rangle^{k}\ut(-\us)\\
&=(-1)^{4k+2}2i\langle\ux,-\ut+i\us\rangle^{k}\langle\ux,\ut+i\us\rangle^{k}\langle\uy,-\ut+i\us\rangle^{k}\langle\uy,\ut+i\us\rangle^{k}\ut\us\\
&=-f_2(\us).
\end{align*}
Hence $f_2(\us)$ is an odd function, therefore 
\[
\int_{\mS^{m-2}} f_2(\us) dS(\us) = 0.
\]
Consequently we have
\begin{align*}
\frac{1}{A_m A_{m-1}}\int_{\mS^{m-1}}\int_{\mS^{m-2}}& \langle\ux,\uta\rangle^{k}\langle\ux,\utd\rangle^{k}\langle\uy,\uta\rangle^{k}\langle\uy,\utd\rangle^{k}\utd\uta dS(\us)dS(\ut)\\
 &= \frac{2}{A_m A_{m-1}}\int_{\mS^{m-1}}\int_{\mS^{m-2}} \langle\ux,\uta\rangle^{k}\langle\ux,\utd\rangle^{k}\langle\uy,\uta\rangle^{k}\langle\uy,\utd\rangle^{k}dS(\us)dS(\ut)\\
&= 2L_{k,k}(\ux,\uy)
\end{align*}
with $L_{k,k}(\ux,\uy)$ the integral defined in Corollary \ref{cor}. Thus for (\ref{kern sum 2}) we have the following result:
\begin{proposition}\label{prop coeffi vartheta}
The dual Radon transform of a term of (\ref{kern sum 2}) is a linear combination of the zonal spherical monogenics, i.e.
\begin{align*}
\frac{1}{A_m A_{m-1}}\int_{\mS^{m-1}}\int_{\mS^{m-2}} \langle\ux,\uta\rangle^{k}\langle\ux,\utd\rangle^{k}&\langle\uy,\uta\rangle^{k}\langle\uy,\utd\rangle^{k}\utd\uta dS(\us)dS(\ut)\\
=& \sum_{j=0}^{2k} \vartheta_{j,k} \ux^{j} \cC_{m,2k-j}(\ux,\uy)\uy^{j}
\end{align*}
where
\begin{align*}
\vartheta_{2j,k} &= 2\theta_{j,k,k}\\
&= \frac{(m-2)(k!)^4 (2k-2n)! \Gamma(m-1)\Gamma\left(\frac{m}{2}+k-n-1\right)^2}{2((k-n)!)^2(n!)^2\Gamma\left(\frac{m}{2}+2k-n\right)^2 \Gamma(2k-2n+m-2)}\\
\vartheta_{2j+1,k} &= -\vartheta_{2j,k}.
\end{align*}
\end{proposition}

\subsection{The inversion}
Finally, we can use the reproducing properties of the zonal spherical monogenics to find an inversion formula.

\begin{proposition}\label{prop::reproducing even case}
Let $M_{t-2n}\in \cM_{t-2n}(\mC^m) $, $t,n\in\mN$ and $t-2n\geq 0$. Then

\[
\tilde{R}[\cR^H_{\uta}[\uz^{2n} M_{t-2n}(\uz)] = \rho_{t,2n} \uz^{2n} M_{t-2n}(\uz)
\]
where

\[
\rho_{t,2n} = \left\{\begin{array}{ll}
\sum_{s=n}^{\ell} \nu_{2\ell-2s,s} \theta_{n,2\ell-s,s}-\frac{1}{4}\nu_{0,\ell}\vartheta_{2n,\ell} & \mathrm{if\ } t = 2\ell \mathrm{ \ even}\\
\sum_{s=n}^{\ell} \nu_{2\ell+1-2s,s} \theta_{n,2\ell+1-s,s} & \mathrm{if\ } t = 2\ell+1 \mathrm{ \ odd}
\end{array}\right.
\]
with $\nu_{k,s} = (-1)^k \frac{\Gamma\left(k+2s+\frac{m}{2}\right)}{\Gamma(k+2s+1)\Gamma\left(\frac{m}{2}\right)}$.
\end{proposition}
\begin{proof}
We have
\begin{align}
\tilde{R}[\cR^H_{\uta}[\uz^{2n} M_{t-2n}(\uz)] =& \tilde{R}\left[\int_{0}^\pi \int_{\mS^{m-1}} L_{\uta}(\uz,e^{i\theta}\uom)(e^{i\theta}\uom)^{2n} M_{t-2n}(e^{i\theta}\uom)dS(\uom)d\theta\right]\nonumber\\
=& \tilde{R}\left[\int_{0}^\pi \int_{\mS^{m-1}}\frac{1}{\pi A_m}\sum_{s=0}^{\infty}\sum_{k=0}^{\infty} \nu_{k,s} a^{k+s}b^{s} (e^{i\theta}\uom)^{2n} M_{t-2n}(e^{i\theta}\uom)dS(\uom)d\theta\right.\nonumber\\
&- \left.\frac{1}{\pi A_m}\frac{\utd\uta}{4}\sum_{s=0}^\infty \nu_{0,s}a^s b^s (e^{i\theta}\uom)^{2n} M_{t-2n}(e^{i\theta}\uom)dS(\uom)d\theta \right]\nonumber\\
=&\int_{0}^\pi \int_{\mS^{m-1}}\frac{1}{\pi A_m}\sum_{s=0}^{\infty}\sum_{k=0}^{\infty} \nu_{k,s}  \tilde{R}\left[a^{k+s}b^{s}\right] (e^{i\theta}\uom)^{2n} M_{t-2n}(e^{i\theta}\uom)dS(\uom)d\theta\label{integ Lie sphere 1}\\
&- \left.\frac{1}{\pi A_m}\sum_{s=0}^\infty \nu_{0,s}\tilde{R}\left[a^{s}b^{s}\frac{\utd\uta}{4}\right] (e^{i\theta}\uom)^{2n} M_{t-2n}(e^{i\theta}\uom)dS(\uom)d\theta \right]\label{integ Lie sphere 2}
\end{align}
where $a=\langle\uz,\uta\rangle\langle e^{-i\theta}\uom,\utd\rangle$,  $b=\langle\uz,\utd\rangle\langle e^{-i\theta}\uom,\uta\rangle$. Thus upon applying Corollary \ref{cor} and Proposition \ref{unicity harmonic}, we get

\begin{align*}
\tilde{R}\left[a^{k+s}b^{s}\right] = \sum_{n'=0}^{s} &\theta_{n', k+s,s} \uz^{2n'} \left[\cC_{m,2s+k-2n'}(\uz,e^{-i\theta}\uom)\right.\\
& \left.- \uz \cC_{m,2s+k-2n'-1}(\uz,e^{-i\theta}\uom)e^{-i\theta}\uom\right] \left(e^{-i\theta}\uom\right)^{2n'}.
\end{align*}
Hence the terms of (\ref{integ Lie sphere 1}) become

\begin{align}
&\int_{0}^\pi \int_{\mS^{m-1}} \uz^{2n'} \cC_{m,2s+k-2n'}(\uz,e^{-i\theta}\uom)\left(e^{-i\theta}\uom\right)^{2n'} (e^{i\theta}\uom)^{2n} M_{t-2n}(e^{i\theta}\uom) dS(\uom)d\theta\nonumber\\
&- \int_{0}^\pi \int_{\mS^{m-1}} \uz^{2n'+1} \cC_{m,2s+k-2n'-1}(\uz,e^{-i\theta}\uom)\left(e^{-i\theta}\uom\right)^{2n'+1} \left(e^{i\theta}\uom\right)^{2n}M_{t-2n}(e^{i\theta}\uom) dS(\uom)d\theta\nonumber\\
=&(-1)^{n'+n}\uz^{2n'}\int_{0}^\pi e^{i\theta(t-2s-k)}d\theta\int_{\mS^{m-1}}  \cC_{m,2s+k-2n'}(\uz,\uom) M_{t-2n}(\uom) dS(\uom) \label{int repr prop 1} \\
&-  (-1)^{n'+n}\uz^{2n'+1}\int_{0}^\pi e^{i\theta(t-2s-k)}d\theta\int_{\mS^{m-1}} \cC_{m,2s+k-2n'-1}(\uz,\uom)\uom M_{t-2n}(\uom) dS(\uom).\label{int repr prop 2}
\end{align}
If we now use the reproducing properties, (\ref{int repr prop 2}) will vanish and the only situation in which (\ref{int repr prop 1}) will not vanish is when $2s+k-2n' = t-2n$. But now 

\begin{align*}
\int_{0}^\pi e^{i\theta(t-2s-k)}d\theta &= \int_{0}^\pi e^{i\theta(2n-2n')}d\theta\\
&= \left\{\begin{array}{ll}
0 & \mbox{if $n\neq n'$,}\\
\pi & \mbox{if $n = n'$.}
\end{array}\right. 
\end{align*}
This implies that the only nontrivial case will be when $n'=n$ and $2s+k=t$, resulting in:

\begin{align*}
\int_{0}^\pi \int_{\mS^{m-1}}\frac{1}{\pi A_m}&\sum_{s=0}^{\infty}\sum_{k=0}^{\infty} \nu_{k,s}  \tilde{R}\left[a^{k+s}b^{s}\right] (e^{i\theta}\uom)^{2n} M_{t-2n}(e^{i\theta}\uom)dS(\uom)d\theta\\
&= \uz^{2n} M_{t-2n}(\uz)\sum_{s=n}^{\left\lfloor\frac{t}{2}\right\rfloor} \nu_{t-2s,s} \theta_{n,t-s,s}.
\end{align*}
Now looking at (\ref{integ Lie sphere 2}), we can use Proposition \ref{prop coeffi vartheta} which yields

\begin{align*}
\tilde{R}\left[a^{s}b^{s}\frac{\utd\uta}{4}\right] = \frac{1}{4}\sum_{n'=0}^{2s} &\vartheta_{n',s} \uz^{n'} \cC_{m,2s-n'}(\uz,e^{-i\theta}\uom)\left(e^{-i\theta}\uom\right)^{n'} 
\end{align*}
In complete analogy to (\ref{integ Lie sphere 1}), we see that $n' = 2n$ and $2s = t$, i.e. (\ref{integ Lie sphere 2}) will only be non-zero if $t = 2\ell$ is even and in this case we have 
\begin{align*}
-\frac{1}{\pi A_m}&\left.\sum_{s=0}^\infty \nu_{0,s}\tilde{R}\left[a^{s}b^{s}\frac{\utd\uta}{4}\right] (e^{i\theta}\uom)^{2n} M_{t-2n}(e^{i\theta}\uom)dS(\uom)d\theta \right]\\
&= -\frac{1}{4}\nu_{0,\ell}\vartheta_{2n,\ell}\uz^{2n} M_{t-2n}(\uz).
\end{align*}

\end{proof}
\begin{proposition}\label{prop::reproducing odd case}
Let $M_{t-2n-1}\in \cM_{t-2n-1}(\mC^m) $, $t,n\in\mN$ and $t-2n-1\geq 0$. Then

\[
\tilde{R}[\cR^H_{\uta}[\uz^{2n+1} M_{t-2n-1}(\uz)] = \rho_{t,2n+1} \uz^{2n+1} M_{t-2n-1}(\uz)
\]
where

\[
\rho_{t,2n+1} = \left\{\begin{array}{ll}
\sum_{s=0}^{\ell} \nu_{2\ell-2s,s} \theta_{n,2\ell-s,s}-\frac{1}{4}\nu_{0,\ell}\vartheta_{2n,\ell} & \mbox{if $t = 2\ell$ even}\\
\sum_{s=0}^{\ell} \nu_{2\ell+1-2s,s} \theta_{n,2\ell+1-s,s} & \mbox{if $t = 2\ell+1$ odd}
\end{array}\right.
\]
with $\nu_{k,s} = (-1)^k \frac{\Gamma\left(k+2s+\frac{m}{2}\right)}{\Gamma(k+2s+1)\Gamma\left(\frac{m}{2}\right)}$.
\end{proposition}
\begin{proof}
The proof is similar to the proof of Proposition \ref{prop::reproducing even case}.
\end{proof}
The last step in the inversion is to write the coefficients $\rho_{t,2n}$ and $\rho_{t,2n+1}$ as operators. We can do this using the Euler operator $\mE_{\uz}$, and the Gamma operator $\Gamma_{\uz}$, using (\ref{eq::Euler,Gamma}) on page \pageref{eq::Euler,Gamma}. Thus we can write $\rho_{t,2n} = \rho_{\mathbb{E}_{\uz},\Gamma_{\uz} + \mE_{\uz}}$ and $\rho_{t,2n+1} = \rho_{\mathbb{E}_{\uz},-\Gamma_{\uz} + \mE_{\uz}+m-1}$ and hence the inversion of the polarized Hua-Radon transform is given by:

\begin{align*}
\uz^{2n} M_{t-2n}(\uz) &= \rho_{\mE_{\uz},\Gamma_{\uz} + \mE_{\uz}}^{-1} \tilde{R}\left[\cR_{\uta}^{H}[\uz_1^{2n} M_{t-2n}(\uz_1)]\right](\uz)\\
\uz^{2n+1} M_{t-2n-1}(\uz) &= \rho_{\mE_{\uz},-\Gamma_{\uz} + \mE_{\uz}+m-1}^{-1} \tilde{R}\left[\cR_{\uta}^{H}[\uz_1^{2n+1} M_{t-2n-1}(\uz_1)]\right](\uz)
\end{align*}
Analogously to the Hua-Radon transform, we can use (\ref{eq::Almansi}), so that we can invert any holomorphic function.

\begin{theorem}\label{inversion polarized}
Let $f(\uz) = \sum_{k=0}^{\infty} M_k(\uz) + \uz\sum_{\ell = 0}^{\infty} N_\ell(\uz) \in\mathcal{OL}^2(LB(0,1))$, with $M_k \in\cM_k(\mC^m)$, $N_{\ell}\in\cM_{\ell}(\mC^m)$, then
\[
f(\uz) = \rho_{\mathbb{E}_{\uz},\Gamma_{\uz} + \mE_{\uz}}^{-1}\tilde{R}\left[\mathcal{R}_{\uta}\left[\sum_{k=0}^{\infty} M_k(\uz')\right]\right](\uz) + \rho_{\mathbb{E}_{\uz},\mE_{\uz}-\Gamma_{\uz}+m-1}^{-1}\tilde{R}\left[\mathcal{R}_{\uta}\left[\uz'\sum_{\ell = 0}^{\infty} N_\ell(\uz')\right]\right] (\uz).
\]
\end{theorem}

\section{Conclusions}
\setcounter{equation}{0}
In this paper we have studied the Hua-Radon and polarized Hua-Radon transform and their inversions. We have proven the unicity of the zonal spherical harmonics with respect to certain properties concerning their homogenicity, harmonicity and spin-invariance. This has lead to an inversion formula for the Hua-Radon transform in Theorem \ref{inversion Hua-Radon} and the polarized Hua-Radon transform in Theorem \ref{inversion polarized} using the techniques shown in \cite{Szego}.


\end{document}